\documentclass[a4paper,12pt]{article}
\usepackage[utf8]{inputenc}
\usepackage[T1]{fontenc}
\usepackage[numbers]{natbib}
\usepackage[english]{babel}
\usepackage{amssymb}
\usepackage{dsfont}
\usepackage{enumerate}
\usepackage[inline]{enumitem}
\usepackage{mathtools}
\usepackage{amsmath}
\usepackage{color}
\usepackage{bbm}         
\usepackage{amsfonts}
\usepackage{graphicx,verbatim}
\usepackage{amsthm}
\usepackage{url}
\usepackage{tikz,multicol,manfnt}
\usepackage{qtree}
\usepackage[weather,clock,alpine]{ifsym}
\usepackage{xparse}
\usepackage{parnotes}
\usepackage{placeins}
\usepackage{xspace}
\usepackage{titling}
\usepackage{xcolor}
\usepackage{tikz-cd}
\usetikzlibrary{positioning}
\usetikzlibrary{decorations.pathreplacing}
\usetikzlibrary{calc,arrows,shapes,fit}
\newcommand{\tikzmark}[1]{%
\tikz[overlay,remember picture,inner sep=0pt, outer sep=0pt] \node (#1) {};}
\usepackage{hyperref}
\usepackage{cleveref}

\newcommand{\rond}{\mathcal}
\newcommand{\urg}[1]{\textcolor{red}{#1}}

\newcommand{\satisf}{\vDash}
\newcommand{\et}{\wedge}

\renewcommand{\phi}{\varphi}
\renewcommand{\epsilon}{\varepsilon}

\newcommand{\geqcur}{\succcurlyeq}
\NewDocumentCommand{\set}{mg}{\left\{#1\IfNoValueF{#2}{\;\middle\vert\;#2}\right\}}

\newcommand{\Qp}{\mathds{Q}_p}
\newcommand{\Fp}{\mathds{F}_p}
\newcommand{\IP}[1]{IP$_{\!#1}$\xspace}
\newcommand{\NIP}[1]{N\IP{#1}}
\newcommand{\IPn}{\IP n}
\newcommand{\NIPn}{\NIP n}
\newcommand{\tpp}{\!\cdots\!}
\newcommand{\ignore}[1]{}

\newcommand{\subtitle}[1]{%
  \posttitle{%
    \par\end{center}
    \begin{center}\large#1\end{center}
    \vskip0.5em}%
}
\newcommand{\subtitles}[3]{%
  \posttitle{%
    \par\end{center}
    \begin{center}\large#1\\\normalsize#2\\\small#3\end{center}
    \vskip0.5em}%
}

\DeclareMathOperator{\Gal}{Gal}

\DeclareMathOperator{\id}{id}       

\DeclareMathOperator{\ch}{ch}

\theoremstyle{plain}
\newtheorem{fct}{Fact}
\newtheorem{thm}{Theorem}[section]
\newtheorem*{thm*}{Theorem}
\newtheorem{cor}[thm]{Corollary}
\newtheorem*{cor*}{Corollary}
\newtheorem{lem}[thm]{Lemma}
\newtheorem{prop}[thm]{Proposition}
\newtheorem*{prop*}{Proposition}
\newtheorem{con}[thm]{Conjecture}

\newtheorem{por}[thm]{Porism}
\theoremstyle{definition}
\newtheorem{dfn}[thm]{Definition}
\theoremstyle{remark}
\newtheorem{rem}[thm]{Remark}
\newtheorem{ex}[thm]{Example}

\begin{document}
 
  \title{Artin-Schreier extensions and combinatorial complexity in henselian valued fields}
  \date{\today}
  \author{Blaise \textsc{Boissonneau}\thanks{The author was funded by Franziska Jahnke's fellowship from the Daimler and Benz foundation. This research was also partially funded by the DAAD through the `Kurzstipendien für Doktoranden 2020/21', and the MSRI via the Decidability, Definablility and Computability programme.}}
  
  \maketitle
  \begin{abstract}
   We give explicit formulas witnessing IP, \IPn or TP2 in fields with Artin-Schreier extensions. We use them to control $p$-extensions of mixed characteristic henselian valued fields, allowing us most notably to generalize to the \NIPn context one way of Anscombe-Jahnke's classification of NIP henselian valued fields. As a corollary, we obtain that \NIPn henselian valued fields with NIP residue field are NIP. We also discuss tameness results for NTP2 henselian valued fields.
  \end{abstract}
  
  \section{Introduction}
  
  This paper started with a question: we know by \cite{KSW} that $\Fp((\Gamma))$ has IP, since it has an Artin-Schreier extension; but what formula witnesses it? We answer this question for IP, \IPn and TP2, see \Cref{locKSWH} and \Cref{locCKS}:
  
  \begin{thm}\label{locintro}
   Let $K$ be an infinite field of characteristic $p>0$. Then $$\phi(x,y_1,\tpp,y_n):\exists t\;x=y_1\tpp y_n(t^p-t)$$ has \IPn iff $K$ has an Artin-Schreier extension, and $$\psi(x,yz):\exists t\; x+z=y(t^p-t)$$ has TP2 iff it has infinitely many distinct Artin-Schreier extensions.   
   \end{thm}
   
   We can use this formula to witness complexity in henselian valued fields of mixed characteristic, allowing us to prove that \NIPn henselian valued fields obey the same conditions than NIP fields (see \cite{AJ-NIP}):
   
\begin{thm}\label{AJ1wayintro}
Let $(K,v)$ be a $p$-henselian valued field. If $K$ is \NIPn, then either:
\begin{enumerate}
\item\label{forbelowintr}$(K,v)$ is of equicharacteristic and is either trivial or SAMK, or
\item $(K,v)$ has mixed characteristic $(0,p)$, $(K,v_p)$ is finitely ramified, and $(k_p,\overline{v})$ satisfies condition \ref{forbelowintr} above, or
\item $(K,v)$ has mixed characteristic $(0,p)$ and $(k_0,\overline{v})$ is AMK.
\end{enumerate}
\end{thm}

Combining it with the original result by Sylvy Anscombe and Franziska Jahnke from \cite{AJ-NIP}, this gives, among others, the following corollary:

\begin{cor}
 Let $(K,v)$ be a \NIPn henselian valued field. If $k$ is NIP, then $(K,v)$ is NIP.
\end{cor}

As for NTP2 henselian valued fields, we prove in \Cref{sec-ntp2}, using again explicit formulas, that NTP2 henselian valued fields obey strong tameness conditions:

\begin{prop}
 Let $K$ be NTP2 and $v$ be $p$-henselian. Then $(K,v)$ is either
 \begin{enumerate}
  \item of equicharacteristic 0, hence tame, or
  \item of equicharacteristic $p$ and semitame, or
  \item of mixed characteristic with $(k_0,\overline v)$ semitame, or
  \item of mixed characteristic with $v_p$ finitely ramified and $(k_p,\overline v)$ semitame.
 \end{enumerate}
 In particular, $(K,v)$ is gdr.
\end{prop}
  
  \subsection{Combinatorial complexity}
  Dating back to the 70's and the work of Saharon Shelah in \cite{Sh78}, model theorists have found that more often than not, meaningful dividing lines between somewhat easy-to-study theories and more complex ones can be expressed in terms of combinatorial configurations that may or may not be encoded in these theories. The prototypical example of this phenomenon is stability: at first studied in terms of the number of different types a theory can have, an equivalent definition is to say that stable theories can not encode an infinite linear order.
  
  This global-local duality between the behavior of the whole theory and the combinatorial properties of individual formulas gives rise to different approaches to study these notions of complexity. One of these approaches is to study the links with algebraic structures. This goes both ways: given an algebraic structure, we want to know how complex it is, a contrario, if we know that some structure has a certain complexity, we want to describe it algebraically.
  
  We like to think about all these notions as a ladder that we try to climb in order to understand theories which are more and more complex. A nice example of this ladder-climbing is the study of Artin-Schreier extensions, which starts in 1999 with the following remarkable result:
   
   \begin{fct}[\cite{Sca-stbl}]Infinite stable fields of characteristic $p>0$ have no Artin-Schreier extensions.\end{fct}
   
   It is in fact conjectured that infinite stable fields have no separable extensions whatsoever; this result tells us that, in characteristic $p$, they at least have no separable extension of degree $p$.
   
   In 2011, this result was pushed up the ladder:
   
   \begin{fct}[\cite{KSW}]Infinite NIP fields of characteristic $p>0$ have no Artin-Schreier extensions; simple fields of characteristic $p>0$ have finitely many distinct Artin-Schreier extensions.\end{fct}
   
   We see here a good example of ladder-climbing; starting with a result in the stable context, it can be extended, sometimes exactly as it is, sometime to a slightly weaker result.
   
   But the ladder continues:
   
   \begin{fct}[\cite{CKS}] NTP2 fields of characteristic $p>0$ have finitely many distinct Artin-Schreier extensions.\end{fct}
   
   \begin{fct}[\cite{H-KSWn}] Infinite \NIPn fields of characteristic $p>0$ have no Artin-Schreier extensions.\end{fct}
   
   We will study in detail those results, explaining the proof strategy, and reduce them to one formula, see \Cref{locintro}.
   
   \subsection{Complexity of henselian valued fields}
   
   In the spirit of the cornerstone AKE transfer principle, transfer theorems have been established in different settings. They are of the form ``if we know enough about the residue field and the value group, then we also know a lot about the valued field''.
   
   NIP transfer theorems have been established as early as 1980, and little by little in more and more cases. They culminated in 2019, with Anscombe-Jahnke's classification of NIP henselian valued fields, that we repeat here:
   
\begin{thm}[Anscombe-Jahnke, \cite{AJ-NIP}]
Let $(K,v)$ be a henselian valued field. Then $(K,v)$ is NIP iff the following holds:
\begin{enumerate}
\item $k$ is NIP, and
\item either
\begin{enumerate}
\item\label{forbelow3} $(K,v)$ is of equicharacteristic and is either trivial or SAMK, or
\item $(K,v)$ has mixed characteristic $(0,p)$, $(K,v_p)$ is finitely ramified, and $(k_p,\overline{v})$ checks \ref{forbelow3}, or
\item $(K,v)$ has mixed characteristic $(0,p)$ and $(k_0,\overline{v})$ is AMK.
\end{enumerate}
\end{enumerate}
\end{thm}

This is as good as it can get; since it is an equivalence, establishing NIP transfer theorems in cases outside of this list is not needed.

Now that we know what the optimal NIP transfer theorem is, we aim to push it up the ladder. There are two directions in this theorem; we study left-to-right (what can be deduced from \NIPn/NTP2) in this paper and will study right-to-left (\NIPn/NTP2 transfer) in a follow-up paper. Some key ingredients of the proof have already been pushed up, most notably the Artin-Schreier closure of NIP fields, which we already mentioned.

One other key ingredient is Shelah's expansion theorem, which fails wildly outside of NIP theories. It is used in mixed characteristic together with the following decomposition: 

\begin{dfn}[Standard Decomposition]\label{standec}
 Let $(K,v)$ be a valued field of mixed characteristic. The standard decomposition around $p$ is defined by fixing two convex subgroups:
 
 $$\Delta_0=\bigcap_{\mathclap{\substack{v(p)\in\Delta\\\Delta\subset\Gamma\text{ convex}}}}\Delta\phantom{MM}\&\phantom{MM}\Delta_p=\bigcup_{\mathclap{\substack{v(p)\notin\Delta\\\Delta\subset\Gamma\text{ convex}}}}\Delta$$
 
 And performing the following decomposition, written in terms of residue maps with specified value groups:
 $$K\xrightarrow{\Gamma_v/\Delta_0}k_0\xrightarrow{\Delta_0/\Delta_p}k_p\xrightarrow{\Delta_p}k_v$$
\end{dfn}

We immediately remark that $\Delta_0/\Delta_p$ is of rank 1 and that $\ch(k_0)=0$ and $\ch(k_p)=p$.

This decomposition is externally definable, thus, adding it to the structure preserves NIP by Shelah's expansion theorem. We can then argue part by part to obtain the result.

It is however possible to bypass this argument: instead of trying to prove that each part is NIP, we can use the explicit formula witnessing IP in fields with Artin-Schreier extensions, and lift complexity to the field. This way, there's no need to add intermediate valuations to the language, at least to prove that relevant part are $p$-closed or $p$-divisible.

This strategy can then be adapted to \NIPn and to NTP2 henselian valued fields. We thus generalize one way of Anscombe-Jahnke to \NIPn fields, see \Cref{AJ1wayintro}, and we prove that NTP2 henselian valued fields obey tameness conditions.

Many thanks to Sylvy Anscombe, Artem Chernikov, Philip Dittmann, Nadja Hempel, Franziska Jahnke, Pierre Simon and Pierre Touchard for their helpful comments.

\section{NIP fields}

We summarize the proof of the following result by Itay Kaplan, Thomas Scanlon and Frank Wagner:

\begin{thm}[\cite{KSW}]\label{NIPAS}
 Infinite NIP fields of characteristic $p$ are Artin-Schreier closed.
\end{thm}
\begin{proof}[Proof summary]
In a NIP theory, definable families of subgroups check a certain chain condition, namely, Baldwin-Saxl's. In an infinite field of characteristic $p>0$, the family $\set{a\wp(K)}{a\in K}$, where $\wp(X)$ is the Artin-Schreier polynomial $X^p-X$, is a definable family of additive subgroups; thus it checks Baldwin-Saxl, and this is only possible if $\wp(K)=K$. The complexity of this argument is mainly hidden in the very last affirmation, we refer to the original paper for details.
\end{proof}
\subsection{Baldwin-Saxl's condition}
We fix a complete theory $T$ and a monster $\mathds M\satisf T$.
\begin{dfn}\label{defNIP}
  A formula $\phi(x,y)$ is said to have the independence property (IP) if there are $(a_i)_{i<\omega},(b_J)_{J\subset\omega}$ such that $\mathds M\satisf\phi(b_J,a_i)$ iff $i\in J$.
 
 A formula is said to be NIP if it doesn't have IP, and a theory is called NIP if all formulas are NIP.
\end{dfn}

Let $(G,\cdot)$ be a group contained, as a set, in $\mathds M$. We do not assume it is definable.

Let $\phi(x,y)$ be an $\rond{L}$-formula such that for any $a\in\mathds M$, $H_a=\phi(M,a)$ is a subgroup of $G$. 

\begin{prop}[Baldwin-Saxl]
  $\phi$ is NIP iff the family $(H_a)_{a\in \mathds M}$ checks the BS-condition: there is $N<\omega$ (depending only on $\phi$) such that for any finite $B\subset \mathds M$, there is a $B_0\subset B$ of size $\leqslant N$ such that:
  $$\bigcap_{a\in B} H_a=\bigcap_{a\in B_0} H_a$$
  That is, the intersection of finitely many $H$'s is the intersection of at most $N$ of them.
\end{prop}

This is a classical result first studied in \cite{BS}. Modern versions can be found in many model theory textbooks, for example \cite{PS-NIP}; however, it is usually not stated as an equivalence, since ``in a NIP theory, all definable families of groups check a specific chain condition'' is much more useful than ``if a specific family checks this hard-to-check chain condition, a specific formula is NIP, but some others might have IP''. We give a proof here for convenience.

\begin{proof}\ 
\paragraph{$\Rightarrow$:}
Assume $\phi$ is NIP, and suppose that the family $(H_a)_{a\in\mathds M}$ fails to check the BS-condition for a certain $N$, that is, we can find $a_0,\!\cdots\!,a_N\in\mathds M$ such that:
  $$\bigcap_{0\leqslant i\leqslant N} H_i\subsetneq\bigcap_{0\leqslant i\leqslant N\,\&\,i\neq j} H_i$$
  for all $j\leqslant N$, and where we write $H_i$ for $H_{a_i}$. We take $b_j\notin H_j$ but in every other $H_i$ and we define $b_I=\prod_{j\in I} b_j$, where the product denote the group law of $G$ -- the order of operations doesn't matter. We have $\mathds M\satisf \phi(b_I,a_i)$ iff $i\notin I$. Because $\phi$ is NIP, there is a maximal such $N$, and thus the BS-condition is checked for some $N$ big enough.
  
  \paragraph{$\Leftarrow$:} Suppose that $(H_a)_{a\in\mathds M}$ checks the BS-condition for a given $N$, and suppose that we can find $a_0,\!\cdots\!,a_N\in A$ and $(b_I)_{I\subset \set{0,\!\cdots\!,N}}\in G$ such that $\mathds M\satisf\phi(b_I,a_i)$ iff $i\in I$. Now by BS, $\bigcap_{0\leqslant i\leqslant N} H_i=\bigcap_{0\leqslant i<N} H_i$ (maybe reindexing it). But now, let $b=b_{\set{0,\!\cdots\!,N-1}}$; we know that $\mathds M\satisf\phi(b,a_i)$ for $i<N$, which means that $b\in \bigcap_{0\leqslant i<N} H_i$, thus $b\in H_N$, and thus $\mathds M\satisf\phi(b,a_N)$, which contradicts the choice of $a$ and $b$.
\end{proof}

\subsection{Artin-Schreier closure and local NIPity}
We can now state the original result by Kaplan-Scanlon-Wagner as an equivalence:
\begin{cor}[Local KSW]\label{locKSW}
 In an infinite field $K$ of characteristic $p>0$, the formula $\phi(x,y)\colon\exists t\;x=y(t^p-t)$ is NIP iff $K$ has no AS-extension.
\end{cor}

\begin{proof}
    Apply previous result with $(G,\cdot)=(K,+)$ and $\phi$ as given: $\phi$ is NIP iff the family $H_a=a\wp(K)$ checks the BS-condition. This then implies that $K$ is AS-closed as discussed in the paragraph following \Cref{NIPAS}. The opposite direction is quite trivial: if $K$ is AS-closed, then $\wp(K)=K$, so the BS-condition is obviously checked.
\end{proof}

\subsection{Lifting}
The formula we obtained says ``this separable polynomial of degree $p$ has a root'', so if it witnesses IP in the residue field of a $p$-henselian valued field, we can lift this pattern to the field itself.

\begin{lem}
 Let $(K,v)$ be $p$-henselian and suppose $k_v$ is infinite, of characteristic $p$, and not AS-closed; then $K$ has IP as a pure field witnessed by $\phi(x,y)\colon\exists t\,x=(t^p-t)y$.
\end{lem}
\begin{proof}
  By assumption and by \Cref{locKSW}, there are $(a_i)_{i<\omega}$ and $(b_J)_{J\subset\omega}$ such that $k_v\satisf\phi(b_J,a_i)$ iff $i\in J$, that is, $P_{i,J}(T)=a_i(T^p-T)-b_J$ has a root in $k_v$ iff $i\in J$. But by $p$-henselianity, taking any lift $\alpha_i$, $\beta_J$ of $a_i$ and $b_J$, $P_{i,J}(T)=\alpha_i(T^p-T)-\beta_J$ has a root in $K$ iff $i\in J$, thus $K\satisf\phi(\beta_J,\alpha_i)$ iff $i\in J$.
\end{proof}

This lemma gives us an explicit formula witnessing IP in some fields; most interestingly, in valued fields of mixed characteristic. For example, consider $K=\Qp(\sqrt[p]{p},\sqrt[p]{\sqrt[p]{p}},\cdots)$: this valued field has residue $\Fp$ and value group $\mathds Z[\tfrac 1 {p^{\infty}}]$; going to a sufficiently saturated extension, we can find a non-trivial proper coarsening $w$ of the $p$-adic valuation $v_p$ with residue characteristic $p$, thus $(k_w,\overline{v_p})$ is a non-trivial valued field of equicharacteristic $p$ with residue $\Fp$, thus it is not AS-closed, and we apply the previous Lemma to $(K,w)$: $K$ has IP as a pure field.

Let us note that bypassing valuations to witness IP in the pure field is not something surprising, as such a result can be obtained in any henselian field, to the cost of explicitness:

\begin{lem}[Jahnke, \cite{Jahnke}]\label{Fran}
 Let $K$ be NIP and $v$ be henselian, then $(K,v)$ is NIP.
\end{lem}

\begin{cor}\label{NIPlift}
 Let $(K,v)$ be henselian, if $(K,v)$ has IP, then $K$ has IP as a pure field. In particular, if $k$ has IP, $K$ has IP.
\end{cor}

At heart of Jahnke's result is Shelah's expansion theorem, since her strategy was to prove that, in most cases, $v$ is externally definable. We refer to \cite{Jahnke} for details.

So, in fact, the main interest of explicit Artin-Schreier lifting is that it skips Shelah's expansion theorem, which only works for NIP theories; moreover it also allows us to slightly relax the henselianity assumption into $p$-henselianity, but only in the specific case where the IPity comes from Artin-Schreier extensions of some residue field.

\section{\texorpdfstring{\NIPn}{NIPn} fields}

\NIPn theories are the most natural generalization of NIP. They were first defined and studied by Shelah in \cite{Shelah-strdep}. Their behavior is erratic, sometimes very similar to NIP theories, sometimes wildly different.

\begin{dfn}
 Let $T$ be a complete theory and $\mathds M\satisf T$ a monster model. A formula $\phi(x;y_1,\dots,y_n)$ is said to have the independence property of order $n$ (\IPn) if there are $(a^k_i)^{1\leqslant k\leqslant n}_{i<\omega}$ and $(b_J)_{J\subset\omega^n}$ such that $\mathds M\satisf\phi(b_J,a^1_{i_1},\dots,a^n_{i_n})$ iff $(i_1,\dots,i_n)\in J$.
 A formula is said to be \NIPn if it doesn't have \IPn, and a theory is called \NIPn if all formulas are \NIPn. We also write ``strictly \NIPn'' for ``\NIPn and \IP{n-1}''.
\end{dfn}

For any $n\geqslant 2$, strictly \NIPn structures exist; for some of algebraic flavor, let us mention pure groups obtained via the Mekler construction, see \cite{CH-Mekler}, or $n$-linear forms, see \cite{ndep2}. However, strictly \NIPn pure fields are believed not to exist:

\begin{con}\label{NIPncon}
 For $n\geqslant2$, strictly \NIPn pure fields do not exist; that is, a pure field is \NIPn iff it is NIP.
\end{con}

This is for pure fields. Augmenting fields with arbitrary structure -- for example by adding a relation for a random hypergraph -- will of course break this conjecture, however, natural extensions of field structure such as valuations or distinguished automorphisms are believed to preserve it. Let us state this conjecture:
 
\begin{con}\label{NIPnhvcon}
 For $n\geqslant2$, strictly \NIPn henselian valued fields do not exist.
\end{con}

It is clear that \Cref{NIPnhvcon} implies \Cref{NIPncon} since the trivial valuation is henselian; we will in fact later prove that they are equivalent, see \Cref{coreqcon}.

We quote some results which make this conjecture somewhat believable:

\begin{prop}[Duret \cite{PAC}, Hempel \cite{Hem-ndep}]\label{PACNSC}
 Let $K$ be PAC and not separably closed. Then, $K$ has \IPn for all $n$.
\end{prop}

\begin{thm}[Hempel, \cite{H-KSWn}]\label{KSWH}
 Infinite \NIPn fields of characteristic $p$ are Artin-Schreier closed.
\end{thm}

Overall, as soon as interesting results are obtained about or in the context of NIP fields, some people (mostly Nadja Hempel and Artem Chernikov) work hard to sneakily add $_n$ after NIP in these results. They succeed most of the time, though not always taking a straightforward route. \Cref{NIPncon} arose naturally from their work and can be attributed to Hempel, in duo with Chernikov.

Going back to \Cref{KSWH}, as for NIP fields, we want to know the formula witnessing \IPn in infinite fields with Artin-Schreier extensions; and, that is a promise, this time there will be a nice application; namely, \Cref{AJ1way}.

The proof of \Cref{KSWH} is similar to Kaplan-Scanlon-Wagner's argument, as one expects: in a \NIPn theory, definable families of subgroups check a certain analog of Baldwin-Saxl's condition. In characteristic $p$, $\set{a_1\cdots a_n\wp(K)}{\overline a\in K^n}$ is a definable family of additive subgroups. In order for it to check the aforementioned chain condition, we must have $\wp(K)=K$, by a similar argument as before.

\subsection{Baldwin-Saxl-Hempel's condition}
Let $T$ be a complete $\rond{L}$-theory, $\mathds M\satisf T$ a monster. Let $(G,\cdot)$ be a group, with $G$ contained in $\mathds M$.

Let $\phi(x,y_1,\!\cdots\!,y_n)$ be an $\rond{L}$-formula such that for all $(a_1,\tpp,a_n)\in\mathds M$, $H_{a_1,\tpp,a_n}=\phi(M,a_1,\tpp,a_n)$ is a subgroup of $G$.

\begin{prop}[Hempel]
  The formula $\phi$ is said to check the BSH$_n$-condition if there is $N$ (depending only on $\phi$) such that for any $d$ greater or equal to $N$ and any array of parameters $(a^{i}_{j})^{1\leqslant i\leqslant n}_{j\leqslant d}$, there is $\overline{k}=(k_1,\tpp,k_n)\in\set{0,\tpp,N}^n$ such that:
  $$\bigcap_{\overline{j}} H_{\overline{j}}=\bigcap_{\overline{j}\neq\overline{k}} H_{\overline{j}}$$
  with $H_{\overline{j}}=H_{a^1_{j_1},\tpp,a^n_{j_n}}$.
  
  The formula $\phi$ checks the BSH$_n$ condition iff $\phi$ is \NIPn.  
\end{prop}
\begin{proof}
    This is a very natural \NIPn version of Baldwin-Saxl, first stated by Hempel in \cite{H-KSWn}. However, as for Baldwin-Saxl, it is usually not stated as an equivalence. We include a proof for convenience.

\paragraph{$\Leftarrow$:}
  Let $\phi$ be \NIPn, and suppose that the BSH$_n$ condition is not checked for $N$, so one can find $(a^{i}_{j})^{1\leqslant i\leqslant n}_{j\leqslant N}\in A$ such that
  $$\bigcap_{\overline{j}} H_{\overline{j}}\subsetneq\bigcap_{\overline{j}\neq\overline{k}} H_{\overline{j}}$$
  for any $\overline{k}\in\set{0,\tpp,N}^n$. 
  
  We take $b_{\overline{j}}\notin H_{\overline{j}}$ but in every other $H_{\overline{k}}$. Then for any $J\subset\set{0,\tpp,N}^n$, we define $b_J=\prod_{\overline{j}\in J} b_{\overline{j}}$, where the product denotes the group law of $G$ -- the order of operation doesn't matter. We have $\mathds M\satisf \phi(b_J,a^1_{j_1},\tpp,a^n_{j_n})$ iff $b_J\in H_{\overline{j}}$ (by definition of $H$), and it is the case iff $\overline{j}\notin J$. If this were to hold for arbitrarily large $N$, we would have \IPn for $\phi$. Thus, if $\phi$ is \NIPn, there is a maximal such $N$.
  
  \paragraph{$\Rightarrow$:} Suppose that $\phi$ checks the BSH$_n$ condition for $N$, and suppose we can find $(a^{i}_{j})^{1\leqslant i\leqslant n}_{j\leqslant N}\in A$ and $(b_J)_{I\subset\set{0,\!\cdots\!,N}^n}\in G$ such that $\mathds M\satisf\phi(b_J,a^1_{j_1},\tpp,a^n_{j_n})$ iff $\overline{j}\in J$. Now by assumption, there is $\overline{k}$ such that $\bigcap_{\overline{j}} H_{\overline{j}}=\bigcap_{\overline{j}\neq\overline{k}} H_{\overline{j}}$. But now, let $b=b_{J\setminus\set{\overline{k}}}$; we know that $M\satisf\phi(b,a^1_{j_1},\tpp,a^n_{j_n})$ iff $\overline{j}\neq\overline{k}$, which means that $b\in \bigcap_{\overline{j}\neq\overline{k}} H_{\overline{j}}$. But this means $b\in H_{\overline{k}}$, which yields $\mathds M\satisf\phi(b,a^1_{k_1},\tpp,a^n_{k_n})$ and contradicts the choice of $b$.
\end{proof}

\subsection{Artin-Schreier closure of \texorpdfstring{\NIPn}{NIPn} fields}

\begin{cor}[Local KSWH]\label{locKSWH}
 In an infinite field $K$ of characteristic $p>0$, the formula $\phi(x;y_1,\tpp,y_n)\colon\exists t\,x=y_1 y_2\cdots y_n(t^p-t)$ is \NIPn iff $K$ has no AS-extension.
\end{cor}

\begin{proof}
    Apply the previous result with $(G,\cdot)=(K,+)$ and $\phi$ as given: $\phi$ is \NIPn iff the family $H_{a_1,\tpp,a_n}=a_1 a_2 \cdots a_n\wp(K)$ checks the BSH$_n$ condition. This then implies that $K$ is AS-closed, see \cite{H-KSWn} -- again, this is the hard part of the proof. The opposite direction is quite trivial: if $K$ is AS-closed, then $\wp(K)=K$, so the BSH$_n$ condition is obviously checked.
\end{proof}

\subsection{Lifting}
Ideally, we would like a \NIPn version of \Cref{NIPlift}. But this relies on \Cref{Fran}, the proof of which needs Shelah's expansion theorem, which fails in general for \NIPn structures; notably, it fails for the random graph.

However, thanks to the explicit formula obtained before and with the help of $p$-henselianity, we can lift \IPn in the case where it is witnessed by Artin-Schreier extensions:

\begin{lem}
 Suppose $(K,v)$ is $p$-henselian and has a residue field $k$ infinite, of characteristic $p$, and not AS-closed; then $K$ has \IPn witnessed by $\phi(x;y_1,\tpp,y_n)\colon\exists t\,x=y_1\cdots y_n(t^p-t)$.
\end{lem}

\begin{proof}
  By assumption and by \Cref{locKSWH}, there are $(a^i_j)^{1\leqslant i\leqslant n}_{j<\omega}$ and $(b_J)_{J\subset\omega^n}$ such that $k\satisf\phi(b_J,a^1_{j_1},\tpp,a^n_{j_n})$ iff $\overline{j}\in J$, that is, $P_{\overline{j},J}(T)=a^1_{j_1}\cdots a^n_{j_n}(T^p-T)-b_J$ has a root in $k$ iff $\overline{j}\in J$. But by $p$-henselianity, since roots of this polynomial are all simple, taking any lift $\alpha^i_j$, $\beta_J$ of $a^i_j$ and $b_J$, $P_{\overline j,J}(T)=\alpha^1_{j_1}\cdots\alpha^n_{j_n}(T^p-T)-\beta_J$ has a root in $K$ iff $\overline j\in J$, thus $K\satisf\phi(\beta_J,\alpha^1_{j_1},\tpp,\alpha^n_{j_n})$ iff $\overline{j}\in J$.
\end{proof}

So, in this specific case, we don't need the valuation to witness \IPn. This fact will have fruitful applications, most importantly \Cref{AJ1way}.

\subsection{\texorpdfstring{\NIPn}{NIPn} henselian valued fields}\label{sec-AJltr}

Throughout this section, $p$ will always equal the residue characteristic of a valued field. When we say that $(K,v)$ is $p$-henselian, we mean $p$-henselian when $p>0$ and we mean nothing if $p=0$.

Our goal is now to prove the following:

\begin{thm}\label{AJ1way}
Let $(K,v)$ be a $p$-henselian valued field. If $K$ is \NIPn, then either:
\begin{enumerate}
\item $(K,v)$ is of equicharacteristic 0, or
\item\label{forbelow4} $(K,v)$ is of equicharacteristic $p>0$ and is either trivially valued or SAMK, or
\item $(K,v)$ has mixed characteristic $(0,p)$, $(K,v_p)$ is finitely ramified, and $(k_p,\overline{v})$ checks \ref{forbelow4}, or
\item $(K,v)$ has mixed characteristic $(0,p)$ and $(k_0,\overline{v})$ is AMK.
\end{enumerate}
\end{thm}

Let $(K,v)$ be \NIPn (as a valued field). Since the residue field is interpretable in a \NIPn structure, it is also \NIPn. In equicharacteristic 0, there is nothing to prove. We do the equicharacteristic $p$ case in the same way as for NIP fields:

\begin{lem}\label{->pp}
 If $(K,v)$ is \NIPn and of equicharacteristic $p$, then it is SAMK or trivial. We do not assume any henselianity here.
\end{lem}
This is a \NIPn version of {\cite[3.1]{AJ-NIP}}.
\begin{proof}
  If $v$ is trivial, then we're done. Assume not. By \Cref{KSWH}, $K$ is AS-closed; this implies that it has no separable algebraic extension of degree divisible by $p$ (see \cite[4.4]{KSW}). Then it is clearly separably defectless, it has $p$-divisible value group, and AS-closed residue. Remains to prove that the residue is perfect. Suppose $\alpha\in k$ has no $p$\textsuperscript{th}-root in $k$, and consider $X^p-mX-a$, where $v(m)>0$ (but non-zero; remember than $v$ is non-trivial) and where $a$ is a lift of $\alpha$. Then this polynomial has no root, thus $K$ is not AS-closed.
\end{proof}

Now, for the mixed characteristic case, we will follow Anscombe-Jahnke's proof for the most part, except we swap Shelah's expansion for explicit Artin-Schreier lifting; while Anscombe-Jahnke's argument works in arbitrary valued fields, ours rely on lifting and thus can't work if we do not assume at least $p$-henselianity.

\begin{lem}\label{resperf}
 Let $(K,v)$ be a \NIPn $p$-henselian valued field. Then $v$ has at most one coarsening with imperfect residue field. If such a coarsening exists, then $p>0$, and this coarsening is the coarsest coarsening $w$ of $v$ with residue characteristic $p$.
\end{lem}
This is a \NIPn version of {\cite[3.4]{AJ-NIP}}.
\begin{proof} If $p=0$, no coarsening of $v$ has imperfect residue field. Assume $p>0$.
  Let $w$ be a proper coarsening of $v$, name $k_w$ its residue. Suppose $k_w$ is of characteristic $p$. Then $(k_w,\overline{v})$ is a non-trivial equicharacteristic $p$ valued field. If its residue is imperfect, then $k_w$ is not AS-closed by the proof of \Cref{->pp}; then $K$ has \IPn as a pure field by explicit Artin-Schreier lifting.
  
  So, if $v$ has a coarsening with imperfect residue field, this coarsening can't in turn have any proper coarsening of residue characteristic $p$; thus the only coarsening of $v$ that could possibly have imperfect residue is the coarsest coarsening of residue characteristic $p$ (possibly trivial).
\end{proof}

\begin{prop}\label{->0p}
  Let $(K,v)$ be a \NIPn $p$-henselian valued field of mixed characteristic $(0,p)$. Then either \begin{enumerate*}\item\label{discrete}$(K,v_p)$ is finitely ramified and $(k_p,\overline{v})$ is SAMK or trivial, or \item $(k_0,\overline{v})$ is AMK.\end{enumerate*}
\end{prop}
This is a \NIPn version of \cite[3.1]{AJ-NIP}.
\begin{proof}
  Consider $(k_p,\overline v)$. If its valuation is non-trivial, $k_p$ must be AS-closed, otherwise $K$ would have \IPn by explicit Artin-Schreier lifting. So, $(k_p,\overline v)$ is either SAMK or trivial by (the proof of) \Cref{->pp}.
  
  We now make the following case distinction: if $\Delta_0/\Delta_p$ is discrete, then $(K,v_p)$ is finitely ramified, and since we already know that $(k_p,\overline v)$ is SAMK or trivial, case \ref{discrete} holds. Otherwise, $\Delta_0/\Delta_p$ is dense. We go to an $\aleph_1$-saturated extension $(K^*,v^*)$ of $(K,v)$, and redo the standard decomposition there. $\Delta_0^*/\Delta_p^*$ is still dense (see \cite[Lem.~2.6]{AJ-NIP}), and by saturation, it is equal to $\mathds{R}$; in particular, $\Delta_0^*/\Delta_p^*$ is $p$-divisible. Now, as before, if $(k_p^*,\overline{v^*})$ is non-trivial, then it is SAMK. It is clearly non-trivial by saturation, since we assumed $(K,v_p)$ was infinitely ramified. Thus, $(k_0^*,\overline{v^*})$ is Kaplansky. We can state this in first order by saying that $k$ is perfect and AS-closed (the valuation $v$ is in our language for now), and that $\Gamma$ is roughly $p$-divisible, i.e. if $\gamma\in[0,v(p)]\subset\Gamma$, then $\gamma$ is $p$-divisible.
  
  Remains to prove that $(k_0,\overline{v})$ is algebraically maximal. First, we prove that $k_p$ is perfect. Consider the $p$-henselian valued field $(K^*,v_p^*)$ (so this time we have $v_p^*$ in the language, and not $v^*$) and an $\aleph_1$-saturated extension $(K',u')$ of it. Since $(K^*,v_p^*)$ is infinitely ramified, by saturation $u'$ admits a proper coarsening of residue characteristic $p$, so by \Cref{resperf}, its residue field is perfect; going down to $(K^*,v_p^*)$, this means $k_p^*$ is perfect. Since we already know that $(k_p^*,\overline{v^*})$ is separably algebraically maximal, because it is perfect we now know it is algebraically maximal.
  
  Now by saturation $(k_0^*,\overline{v_p^*})$ is maximal; in particular it is defectless, see \cite{AK-tame}. Now $v^*$ is a composition of defectless valuations, thus it is defectless (see \cite[Lem.~2.8]{AJ-NIP}). By \cite[Lem.~2.4]{AJ-NIP}, defectlessness is a first-order property, so $(K,v)$ is also defectless, and thus $(k_0,\overline{v})$ is defectless. Because defectlessness implies algebraic maximality, we conclude.
\end{proof}

This Theorem extends half of Anscombe-Jahnke's classification of NIP henselian valued fields. We thus have the following: 
\begin{cor}\label{nipres}
 Let $(K,v)$ be henselian and \NIPn. If $k_v$ is NIP, so is $(K,v)$.
\end{cor}
\begin{proof}
 If $(K,v)$ is henselian, it is in particular $p$-henselian, and so we can apply \Cref{AJ1way} to it. But in all the cases of the theorem, we know that we have NIP transfer by Anscombe-Jahnke's full classification; this means that if $k_v$ is NIP, so is $(K,v)$. We need henselianity and not just $p$-henselianity for transfer to happen.
\end{proof}

\begin{cor}\label{coreqcon}
 \Cref{NIPncon}$\Leftrightarrow$\Cref{NIPnhvcon}; that is, if no strictly \NIPn pure field exist, no strictly \NIPn henselian valued field exist.
 
 In particular, both conjectures hold in algebraic extensions of $\Qp$.
\end{cor}

\begin{proof}
 Indeed, if no strictly \NIPn pure field exist, the residue field of a \NIPn henselian valued field must be in fact NIP, and we conclude by \Cref{nipres}.
 
 Now consider algebraic extensions of $\Fp$. They are either finite, alebraically closed, or PAC and not separably closed; in the first two cases they are NIP, in the last they have \IPn for all $n$. So they are NIP iff they are \NIPn, and any henselian valued field with one of these extensions as residue field is \NIPn iff it is NIP.
 
 Lastly, in any (non-algebraically closed) algebraic extension of $\Qp$, the $p$-adic valuation is definable; thus they are \NIPn as pure fields iff they are \NIPn as a valued fields.
\end{proof}

In a follow-up paper, we will study transfer theorems and complete the proof of Anscombe-Jahnke's classification in the \NIPn context.

\section{NTP2 fields}\label{sec-ntp2}

\subsection{The tree property of the second kind}

\begin{dfn}
 A formula $\phi(x,y)$ is said to have the tree property of the second kind (TP2) if there are $(a_{ij})_{(i,j)\in\omega^2}$ and $k<\omega$ such that for any $i<\omega$, $\set{\phi(x,a_{ij})}{j<\omega}$ is $k$-inconsistent, but for any $f\colon\omega\rightarrow\omega$, $\set{\phi(x,a_{if(i)})}{i<\omega}$ is consistent.
 
 A formula is NTP2 if it doesn't have TP2, and a theory is NTP2 if all its formulas are NTP2.
\end{dfn}

 Note that NIP implies NTP2, but that \NIPn doesn't: the random graph is \NIP2 and NTP2, the triangle-free random graph is \NIP2 and TP2. Also, NTP2 is not preserved under boolean combinations.
 \begin{figure}[!htbp]
   $$\begin{matrix}
   \tikzmark{a00}\phi(x,a_{00}) & \phi(x,\tikzmark{a01}a_{01}) & \cdots^{\ \tikzmark{a02}}\\
   \phi(x,\tikzmark{a10}a_{10}) & \phi(x,a_{11}) & \cdots^{\ }\\
   \vdots & \vdots\tikzmark{a2} &\\
   &\text{\urg{consistent}}&\\
   \end{matrix}
$$
 
\begin{tikzpicture}[overlay,remember picture]
\draw[very thick, red] (a01)--(a10) -- (a2);
\node(e)[fit=(a00)(a02), draw=violet, rounded rectangle, inner sep=4pt] {};
\node[above=0.05cm of e, violet] {$k$-inconsistent};
\end{tikzpicture}
\caption{A TP2 pattern}
\end{figure}
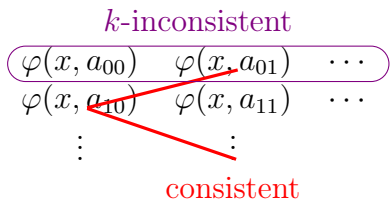
 
 \begin{ex}
 Bounded PAC, PRC and PpC fields are NTP2, see \cite{prc}.
 
 As pure rings, $\mathds Z$ and thus also $\mathds Q$ have TP2: in $\mathds Z$, the formula ``$x$ divides $y$ and $x\neq 1$'' has TP2. However its negation does not, since rows can't be $k$-inconsistent.
\end{ex}

\subsection{NTP2 fields}

\begin{thm}[\cite{CKS}]
 NTP2 fields of characteristic $p$ are AS-finite, also called $p$-bounded -- they have only finitely many distinct Artin-Schreier extensions.
\end{thm}

Chernikov-Kaplan-Simon's argument is very similar to Kaplan-Scanlon-Wagner's. First, one needs to find a suitable chain condition for definable families of subgroups in NTP2 theories, and then apply it to the Artin-Schreier additive subgroup. Namely, instead of saying that the intersection of $N+1$ subgroups is the same as just $N$ of them, this condition is saying that the intersection of all but one of them is not quite the whole intersection, but is of finite index in it. Then, one shows that in a field $K$ with infinitely many Artin-Schreier extensions, the family $a\wp(K)$ fails this condition.

\subsection{Chernikov-Kaplan-Simon condition for NTP2 formulas}

\begin{thm}[{\cite[Lem.~2.1]{CKS}}]
  Let $T$ be NTP2, $\mathds M\satisf T$ a monster and suppose that $(G,\cdot)$ is a definable group\footnotemark. Let $\phi(x,y)$ be a formula, for $i\in\omega$ let $a_i\in\mathds M$ be such that $H_i=\phi(\mathds M,a_i)$ is a normal subgroup of $G$. Let $H=\bigcap_{i\in\omega}H_i$ and $H_{\neq j}=\bigcap_{i\neq j}H_i$. Then there is an $i$ such that $\left[H_{\neq i}\colon H\right]$ is finite.
\end{thm}
\footnotetext{In fact, as before, we do not care whether $G$ is a definable set, however, we need the group law to be definable, as it appears in the formula $\psi$.}

It turns out that, once again, we do not need $T$ to be completely NTP2: the proof goes by contradiction and shows that if this finite index condition is not respected, the formula $\psi(x;y,z)\colon\exists w\, (\phi(w,y)\et x=w\cdot z)$ has TP2. Thus we need only to assume NTP2 for this $\psi$. As in the NIP case for Baldwin-Saxl, we establish an equivalence between one specific formula being NTP2 and this condition.

\begin{rem}
 This condition says that in a given family of subgroups, one of them has finitely many distinct cosets witnessed by elements which lie in the intersection of every other subgroup. By compactness, we can cap this finite number, and consider only finite families: there is $k$ and $N$, depending only on $\phi$, such that given $k$ many subgroups defined by $\phi$, one of them has no more than $N$ cosets witnessed by elements in the intersection of the $k-1$ other subgroups. 
\end{rem}

\begin{por}[CKS-condition for fomulas]\label{CKSeq}
  Let $T$ be an $\rond{L}$-theory, $\mathds M\satisf T$ a monster and $(G,\cdot)$ a definable group. Let $\phi(x,y)$ be a formula such that for any $a\in M$, $H_a=\phi(\mathds M,a)$ is a normal subgroup of $G$. Let $\psi(x;y,z)$ be the formula $\exists w\, (\phi(w,y)\et x=w\cdot z)$. We will suppose for more convenience that $\cdot$, or rather, the formula defining $\set{x,y,z}{x\cdot y = z}$ contains, or at least implies, $x,y,z\in G$; thus $\psi$ doesn't hold if $z\notin G$. Then $\psi(x;yz)$ is NTP2 iff the CKS-condition holds: for any $(a_i)_{i\in\omega}$, there is $i$ such that $[H_{\neq i}\colon H]$ is finite, where $H=\bigcap_{i\in\omega}H_i$ and $H_{\neq j}=\bigcap_{i\neq j}H_i$.
\end{por}

Note that since $^{-1}$ is definable, $\psi(x;y,z)$ is equivalent to $\phi(x\cdot z^{-1},y)$.
\begin{proof}
  The formula $\psi(x;yz)$ holds iff $x\in H_y\cdot z$. Also, we use $H_i$ to denote $H_{a_i}$ and later $H_i^j$ to denote $H_{a_{ij}}$ because it is much more convenient.
  
  We work in four steps, but truly, only the fourth step is an actual proof, and it is technically self-sufficient. The raison d'être of step 1 to 3 is to -- hopefully -- make the proof strategy clearer.
  
 \paragraph{Step 1: true equivalence, from CKS.}
 In their paper, Chernikov, Kaplan and Simon prove that given some $(a_i)_{i\in\omega}$, if the family $H_i$ does not check the CKS-condition, then $\psi$ has TP2. They do this by explicitly witnessing TP2 by $c_{ij}=(a_i,b_{ij})$, with $a$ for $y$ and $b$ for $z$, and with $b_{ij}\in H_{\neq i}$. Reversing their argument, we prove the following equivalence:

 \emph{$\psi$ has TP2 witnessed by some $c_{ij}=(a_i,b_{ij})$ with $b_{ij}\in H_{\neq i}$ iff the family $H_i$ does not check the CKS-condition.}
 
 Right-to-left is exactly given by the original paper. Now let $a_i$ and $b_{ij}$ be as wanted. $\psi(x;c_{ij})$ says that $x\in H_i\cdot b_{ij}$. So the TP2-pattern is as follows:
 
 $$\begin{matrix}
   H_0b_{00} & H_0b_{01} & H_0b_{02} & H_0b_{03} & \cdots \\
   H_1b_{10} & H_1b_{11} & H_1b_{12} & H_1b_{13} & \cdots \\
   H_2b_{20} & H_2b_{21} & H_2b_{22} & H_2b_{23} & \cdots \\
\vdots & \vdots & \vdots & \vdots & \\
   \end{matrix}
$$

 For a given $i$, $k$-inconsistency of the rows says that a given coset of $H_i$ might only appear $k-1$ times. So there are infinitely many cosets of $H_i$, witnessed by elements $b_{ij}\in H_{\neq i}$. This means that $H\cdot b_{ij}=H\cdot b_{ij'}$ iff $H_i\cdot b_{ij}=H_i\cdot b_{ij'}$. But that gives infinitely many cosets of $H$ in $H_{\neq i}$, for any $i$, proving that CKS-condition is not checked.
 
 Note that we did not use at any time consistency of the vertical paths. We can use it to loosen our assumption. Let's keep in mind that our final goal is to prove this equivalence with $a$ depending on $i$ and $j$ (right now it depends only on $i$) and with $b_{ij}$ not necessarily lying in $H_{\neq i}$.
 
 \paragraph{Step 2: going outside $H_{\neq i}$.} We now want to prove:
 
 \emph{$\psi$ has TP2 witnessed by some $c_{ij}=(a_i,b_{ij})$ with iff the family $H_i$ does not check the CKS-condition.}
 
 We already know right-to-left. Let $c_{ij}=(a_i,b_{ij})$ witness TP2 for $\psi$. Consistency of the vertical paths implies that there is $\lambda\in\bigcap_{i\in\omega} H_i\cdot b_{i0}$. Now write $b'_{ij}=b_{ij}\cdot\lambda^{-1}$. Replacing $b$ by $b'$ won't alter TP2, but will ensure that $H_i b_{i0}=H_i$. So we might as well take $b'_{i,0}$ to be the neutral element of $G$.

 Fix $i,j$. Consider the vertical path $f=\delta_{ij}\colon\omega\rightarrow\omega$ such that $\delta_{ij}(i)=j$ and $
 \delta_{ij}(i')=0$ for $i'\neq i$. Consistency yields: $H_i\cdot b'_{ij}\cap \bigcap_{i'\neq i} H_{i'}=H_i\cdot b'_{ij}\cap H_{\neq i}\neq\emptyset$. Thus we can witness this coset of $H_i$ by an element $b''_{ij}\in H_{\neq i}$. Thus $c''_{ij}=(a_i,b''_{ij})$ still witnesses TP2.
 
  $$\begin{matrix}
   H\tikzmark{a0}_0 & H_0b_{01} & \cdots &&\\
   \vdots\tikzmark{a1} & \vdots &&& \\
   H_i & H_ib_{i1} & \cdots & H_i^{\tikzmark{a2}}b_{ij}& \cdots \\
   \vdots^{\tikzmark{a3}} & \vdots & \vdots & \vdots & \\
   \tikzmark{a4}&&&&\\
   \end{matrix}
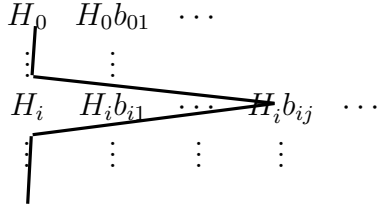
$$
 
\begin{tikzpicture}[overlay,remember picture]
\draw[very thick] (a0)--(a1) -- (a2) --(a3)--(a4);
\end{tikzpicture}
 
 Thus, we reduced to the case in step 1, and we can drop the assumption on $b$. We still have to drop the assumption on $a$. We used $k$-inconsistency of rows in step 1, we used consistency of (some) vertical paths in step 2, we didn't yet use normality.
 
 \paragraph{Step 3: arbitrary $a$, 2-inconsistency.}
 An example of such a TP2 pattern in $\mathds{Z}$:
 
 $$\begin{matrix}
2\mathds{Z} & 4\mathds{Z}+1 & 8\mathds{Z}+3 & 16\mathds{Z}+7 & \cdots \\
3\mathds{Z} & 9\mathds{Z}+1 & 27\mathds{Z}+4 & 81\mathds{Z}+13 & \cdots \\
5\mathds{Z} & 25\mathds{Z}+1 & 125\mathds{Z}+6 & 625\mathds{Z}+31 & \cdots \\
\vdots & \vdots & \vdots & \vdots & \\
   \end{matrix}
$$

Note that none of these subgroups have infinitely many cosets, let alone in the intersection of the others! But, for any $N$, some of them will have more cosets than $N$.

We aim to prove the following, of which once again we know right-to-left:

\emph{There is some $c_{ij}=(a_{ij},b_{ij})$ forming a TP2 pattern for $\psi$, with rows 2-inconsistent, iff the family $H_i$ does not check the CKS-condition.}

Let $H_i^j$ be the subgroup $\phi(M,a_{ij})$. Suppose $\psi$ has TP2, witnessed by $c_{ij}=(a_{ij},b_{ij})$. As noted before, by compactness we do not need to find an infinite family such that every subgroup has infinitely many cosets in the intersection of the rest, but merely for each finite $m$ and $N$, a family of $m$ sugroups such that each of them has at least $N$ cosets in the intersection of the rest.

First, we apply the reduction as before: by consistency of vertical paths, we may take $b_{i0}$ to be the neutral element for each $i$. Then, looking at the path $f=\delta_{ij}$, we may assume $b_{ij}\in H_{\neq i}^0$.

\paragraph{Claim.} Let $N\in\omega$. For each $i$, there is $j$ such that $(b_{ij'})_{j'<\omega}$ witnesses at least $N$ cosets of $H_i^j$: $\#\set{H_i^jb_{ij'}}{j'\in\omega}\geqslant N$.
\par\ 
\par
Before proving this claim, let's see why it is enough for our purpose: let $N\in\omega$. For a fixed $i$, we find $j_i$ such that $H_i^{j_i}$ has $\geqslant N$ cosets witnessed by some $b_{ij}$. Now by vertical consistency, considering the path $\delta_{ij_i}$, we find an element $\lambda\in H_{\neq i}^0\cap H_i^{j_i}b_{ij_i}$. Compose everything by $\lambda^{-1}$, re-index the sequence by switching $c_{i0}$ and $c_{ij_i}$; this makes it so we can assume that $H_i^0$ has $\geqslant N$ many cosets in $H_{\neq i}^0$. When we compose by $\lambda$, nothing changes: $b$ and $b'$ generate the same coset of $H$ iff $b'b^{-1}\in H$ iff $(b'\lambda)(b\lambda)^{-1}\in H$. So we do this row by row, and we might assume that for any $i$, $H_i^0$ has $\geqslant N$ many cosets witnessed by elements from $H_{\neq i}^0$. This implies that some family will fail the CKS condition by compactness.
 
Now to prove the claim, fix $i$ and $N$. If there is $j$ such that $H_i^j$ has infinitely many cosets, witnessed in the row $i$, then we're done. Otherwise, for each $j$, all $H_i^j$ have finitely many cosets. We will reduce the problem in the following way:

$H_i^0$ has finitely many cosets in an infinite row, so by pigeonhole, one of them appears infinitely many times. Ignore all the rest, rename them; we may thus assume that $H_i^0b_{ij}=H_i^0b_{i1}$ for any $j\geqslant1$. We can do the same thing with any $j$, ensuring that $H_i^jb_{ik}=H_i^jb_{i,j+1}$ for any $k>j\in\omega$. Note that we only assume that cosets of a given $H_i^j$ witnessed by $b$ appearing after $j$ are identical, not before, since we already modified things before. In short, we have $b_{ij}b_{ik}^{-1}\in H_i^{j-1}$ for any $i,j$, and $k>j$.

Up to this point, we didn't use 2-inconsistency, so everything will still hold for the $k$-inconsistent case.
 
Because of 2-inconsistency, cosets of $H_i^j$ appearing before $j$ cannot be the same: let $j_1<j_2<j_3$. By our reduction, we have $b_{ij_3}b_{ij_2}^{-1}\in H_i^{j_1}$. Suppose furthermore that $b_{ij_2}b_{ij_1}^{-1}\in H_i^{j_3}$, so 2 cosets of $H_i^{j_3}$ appearing before $j_3$ are the same. Now $b_{ij_3}b_{ij_2}^{-1}b_{ij_1}=(b_{ij_3}b_{ij_2}^{-1})b_{ij_1}\in H_i^{j_1}b_{ij_1}$ on one hand, and $b_{ij_3}b_{ij_2}^{-1}b_{ij_1}=b_{ij_3}(b_{ij_2}^{-1}b_{ij_1})\in b_{ij_3}H_i^{j_3}=H_i^{j_3}b_{ij_3}$ by normality on the other hand, contradicting 2-inconsistency.
 
Thus, if we take $j\geqslant N$, we are sure that $H_i^j$ has $\geqslant N$ many cosets witnessed in the row $i$, proving the claim.
 
\paragraph{Step 4: $k$-inconsistency.} We now are ready to prove \Cref{CKSeq}. We already know one direction, so we now prove that if $\psi$ has TP2 witnessed by some $c_{ij}=(a_{ij},b_{ij})$, then the family $H_i$ does not check the CKS condition.
 
We follow the argument of step 3 until the point where 2-inconsistency enters the party. We aim to prove the claim. First, we fix $i$; since the argument now does not depend on $i$, we stop writing the subscripts $i$; readers attached to formal correctness are invited to take a pen and scribble them back in place.

Let $j_1<j_2<\cdots<j_{2k-1}\in\omega$. Suppose that $b_{j_1}$ and $b_{j_2}$ spawn the same coset of $H^{j_3},H^{j_5},\!\cdots\!,H^{j_{2k-1}}$, so $b_{j_1}b_{j_2}^{-1}\in H^{j_3}\cap H^{j_5}\cap\cdots\cap H^{j_{2k-1}}$. Similarly, suppose $b_{j_3}$ and $b_{j_4}$ spawn the same coset of all the odd indexed groups above them, and again for all the rest. Let $b=b_{j_1}b_{j_2}^{-1}b_{j_3}b_{j_4}^{-1}\cdots b_{j_{2k-3}}b_{j_{2k-2}}^{-1}b_{j_{2k-1}}$. We claim that $b\in H^{j_1}b_{j_1}\cap H^{j_3}b_{j_3}\cap\cdots\cap H^{j_{2k-1}}b_{j_{2k-1}}$, contradicting $k$-inconsistency: Fix $n\in\set{1,3,\!\cdots\!,2k-1}$. By the reduction, all the products $b_jb_{j'}^{-1}$ on the right of $b_{j_n}$ are in $H^{j_n}$, and by assumption, all the products on the left also. Thus $b=hb_{j_n}h'$, where $h,h'\in H^{j_n}$. So $b\in H^{j_n}b_{j_n}H^{j_n}$, and by normality we conclude.
 
Therefore, we know that as soon as $j_1<j_2<\cdots<j_{2k-1}$, there is a pair $b_{j_n}$, $b_{j_{n+1}}$, with odd $n$, that do not spawn the same coset of some $H^{j_n'}$, $j_{n'}>j_{n+1}$. We want to show that some $H^{j_n}$ must have at least $N$ many different cosets, for arbitrary $N\in\omega$.

Fix $N$. Let $j_{2k-1}>C$, where $C$ is a big enough constant we will explicit later. We construct a graph with $N$ vertices, which are the $j$ such that $j_{2k-1}-(N+1)<j<j_{2k-1}$, and $j,j'$ are connected iff $b_j$ and $b_{j'}$ generate \emph{different} cosets of $H^{j_{2k-1}}$. This forces $C\geqslant N$. If it is a complete graph, then $H^{j_{2k-1}}$ has at least $N$ many pairwise disjoint cosets, so we are done. Otherwise, there are $j_{2k-1}-(N+1)<j_{2k-3}<j_{2k-2}<j_{2k-1}$ such that $b_{j_{2k-3}}$ and $b_{j_{2k-2}}$ generate the same coset of $H^{j_{2k-1}}$.
 
We now look back $R_2(N)$ points before $j_{2k-3}$. Here we call $R_r(s)$ the smallest number $V\in\mathds{N}$ such that if a complete colored graph with $r$ many colors has at least $V$ many vertices, there's a monochromatic $s$-clique. $R_r(s)$ is guaranteed to exist for any $r,s\in\mathds N$ by Ramsey's theorem, see \cite{Ramsey}.

Since $j_{2k-3}>C-N$, we take $C\geqslant N+R_2(N)$. We construct a bi-colored graph with $R_2(N)$ vertices, which are the $j$ such that $j_{2k-3}-(R_2(N)+1)<j<j_{2k-3}$. $j,j'$ are connected by a blue edge iff $b_j$ and $b_{j'}$ generate 2 different cosets of $H^{j_{2k-3}}$, and they are connected by a red edge iff they generate different cosets of $H^{j_{2k-1}}$. They might be connected by both a red and blue edge at the same time, this does not break the argument. If you don't like when edges coincide, choose one color arbitrarily. As before, if this graph is complete, then by Ramsey's theorem, there must be a monochromatic $N$-clique, ensuring that one of $H^{j_{2k-1}}$ or $H^{j_{2k-3}}$ have at least $N$ many different cosets. Otherwise, we find a pair $j_{2k-5}<j_{2k-4}$ generating the same coset of both $H^{j_{2k-1}}$ and $H^{j_{2k-3}}$, we fix them, and continue.
 
We now construct a tri-colored graph with $R_3(N)$ vertices, corresponding to the $R_3(N)$ indices preceding $j_{2k-5}$, blue edge between vertices if they generate different cosets of $H^{j_{2k-1}}$, red if they generate different cosest of $H^{j_{2k-3}}$, green if they generate different cosets of $H^{j_{2k-5}}$. Again, by Ramsey's theorem, we either can find an $N$-clique, in which case we stop here, or we can find $j_{2k-7}$ and $j_{2k-6}$ not connected (hence generating the same coset of all of the previously fixed groups). This construction is illustrated in \Cref{tricolored}.

We continue doing this strategy for as long as we can; either we stop when we find a monochromatic $N$-clique, or we end up with $j_1<j_2<\cdots<j_{2k-1}$ such that all consecutive pairs generate the same coset of all subgroups above them; but as seen before, this contradicts $k$-inconsistency. Therefore, this process must stop before, which means we found a clique at some point, and that guarantees a subgroup with at least $N$ many different cosets.
 
 As for the value of $C$, the construction requires $C\geqslant N+R_2(N)+R_3(N)+\cdots+R_{k}(N)$, and any such $C$ works.
\end{proof}

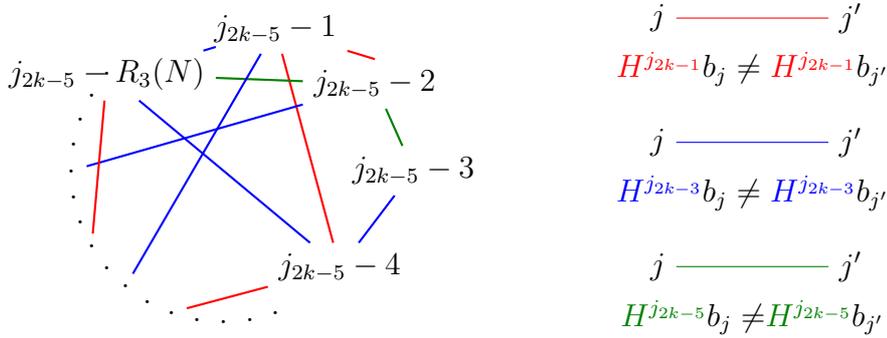
\begin{figure}[!htbp]
 \begin{center}
  \begin{minipage}{0.48\textwidth}
  \begin{center}
  \begin{tikzpicture}
   \foreach \a in {14,15,...,29}{
\draw (\a*10: 2cm) node(p\a){.};
}
\draw (140: 2cm) node(jR){$j_{2k-5}-R_3(N)$};
\draw (70: 2cm)  node(j1){$j_{2k-5}-1$};
\draw (30: 2.3cm)  node(j2){$j_{2k-5}-2$};
\draw (0: 2.5cm)  node(j3){$j_{2k-5}-3$};
\draw (-40: 2cm)   node(j4){$j_{2k-5}-4$};

\draw[thick,red] (j1)--(j4);
\draw[thick,green!50!black] (j2)--(jR);
\draw[thick,green!50!black] (j2)--(j3);
\draw[thick,blue] (j2)--(p18);
\draw[thick,red] (j4)--(p25);
\draw[thick,red] (jR)--(p21);
\draw[thick,blue] (j3)--(j4);
\draw[thick,blue] (jR)--(j4);
\draw[thick,blue] (j1)--(p23);
\draw[thick,red] (j1)--(j2.north);
\draw[thick,blue] (j1)--(jR.north east);
  \end{tikzpicture}
  \end{center}
  \end{minipage}
\begin{minipage}{0.48\textwidth}
\begin{center}
  \begin{tikzpicture}
  \node(j1){$j$};
  \node[right=2cm of j1](j'1){$j'$};
  \draw[-,red] (j1) to node[below=0.3cm]{$H^{j_{2k-1}}\color{black}b_j\neq\color{red}H^{j_{2k-1}}\color{black}b_{j'}$}(j'1);
  
  \node[below=1cm of j1](j2){$j$};
  \node[right=2cm of j2](j'2){$j'$};
  \draw[-,blue] (j2) to node[below=0.3cm]{$H^{j_{2k-3}}\color{black}b_j\neq\color{blue}H^{j_{2k-3}}\color{black}b_{j'}$}(j'2);
  
  \node[below=1cm of j2](j3){$j$};
  \node[right=2cm of j3](j'3){$j'$};
  \draw[-,green!50!black] (j3) to node[below=0.3cm]{$H^{j_{2k-5}}{\color{black}b_j\neq}H^{j_{2k-5}}{\color{black}b_{j'}}$}(j'3);
  \end{tikzpicture}

\end{center}
  \end{minipage}
 \end{center}
 \caption{After finding $j_{2k-5},\tpp,j_{2k-1}$, we connect the $R_3(N)$ many points $j_{2k-5}-1,\tpp,j_{2k-5}-R_3(N)$ with edges colored as indicated; we seek either a monochromatic $N$-clique or two non-connected points that we then name $j_{2k-6}$ and $j_{2k-7}$.}
 \label{tricolored}
\end{figure}

\begin{rem}
 CKS asked whether normality is a necessary assumption. In our proof as well as in theirs, it is useful to assume it, and doesn't seem avoidable. It seems to us that this assumption is necessary, but as of yet, no argument exists to assert or refute this claim.
\end{rem}

\subsection{Artin-Schreier finiteness of NTP2 fields}

\begin{cor}[Local CKS]\label{locCKS}
 In a field $K$ of characteristic $p>0$, the formula $$\psi(x;y,z)\colon\exists t\, x-z=y(t^p-t)$$ is NTP2 iff $K$ has finitely many AS-extensions.
\end{cor}

\begin{proof}
  Apply \Cref{CKSeq} with $(G,\cdot)=(K,+)$ and with $\phi(x,y)\colon\exists t\; x=(t^p-t)y$, which means ``$x\in y\wp(K)$''. If the formula is NTP2 then it checks CKS and thus $K$ has finitely many AS-extensions, by the original CKS argument -- which goes by contraposition, and again, takes a whole paper to be properly done. Now if $K$ has finitely many AS-extensions, then $\left[K\colon\wp(K)\right]$, as additive groups, is finite. Thus any additive subgroup of the form $a\wp(K)$ has finitely -- and boundedly -- many cosets in the whole $K$, so in particular in any intersection of any family. Thus CKS is checked and $\psi$ is NTP2.
\end{proof}

\begin{rem}\label{ditt}
 This is optimal, in the sense that NTP2 fields with an arbitrarily large number of Artin-Schreier extensions exist: given a profinite free group with $n$ generators, there exists a PAC field of characteristic $p$ having this group as absolute Galois group. Such a field will have finitely many Galois extension of each degree, that is, it is bounded and hence simple; but if one takes $n$ large enough, it will have an arbitrarily large number of Artin-Schreier extensions.
 
 On the other hand, fields with finitely many Artin-Schreier extensions can have TP2: consider a PAC field of characteristic $p$ which is unbounded for some $n\neq p$, and take its $p$-closure; still PAC, still unbounded, thus TP2; however, it has no Artin-Schreier-extension.
\end{rem}

We now discuss two applications of local CKS: one is, as for \NIPn, lifting complexity, and the other one is only a potential programme to obtain NTP2 of some fields, most notably, $\Fp((\mathds Q))$.

\subsection{Lifting}
Let $(K,v)$ be $p$-henselian of residue characteristic $p>0$. Shelah's expansion doesn't work in general in NTP2 theories, so adding coarsenings to the language might disturb NTP2. Note however that some weaker versions hold, for example \cite[Annex A]{ShexpNTP2}, where one needs to ensure that the value group is NIP and stably embedded before adding coarsenings to the theory. Meanwhile, we can apply the same trick as above to lift complexity and derive some conditions on NTP2 fields.

\begin{lem}
 Let $(K,v)$ be $p$-henselian of residue characteristic $p$ and suppose $k$ has infinitely many AS-extensions, then $K$ has TP2 witnessed by $\psi(x;y,z)\colon\exists t\; x-z=y(t^p-t)$.
\end{lem}

\begin{proof}
  Since $k$ has infinitely many AS-extensions, we know that there are $(a_{ij},b_{ij})_{i,j<\omega}$ in $k$ witnessing TP2 for $\psi$. Take any lift $\alpha_{ij}$, $\beta_{ij}$ in $K$, we claim that they witness a TP2 pattern for $\psi$ in $K$.
  \paragraph{Vertical consistency:} Let $f\colon\omega\rightarrow\omega$ be a vertical path. We know that there is $c$ in $k$ such that $k\satisf\psi(c;a_{if(i)}b_{if(i)})$ for all $i$.\footnotemark\  This means $a_{if(i)}(T^p-T)-c-b_{if(i)}$ has a root in $f$. Take any lift $\gamma$ of $c$, then $\alpha_{if(i)}(T^p-T)-\gamma-\beta_{if(i)}$ has a root in $K$ by $p$-henselianity, which means $K\satisf\psi(\gamma;\alpha_{if(i)},\beta_{if(i)})$.
  \footnotetext{This is only true if $K$ is $\aleph_1$-saturated, so let's assume it is.}
  \paragraph{Horizontal $m$-inconsistency:} let's name $P_{ij}(T,x)=a_{ij}(T^p-T)-b_{ij}-x$. Now the residue field $k\satisf\psi(c;a_{ij},b_{ij})$ iff $P_{ij}(T,c)$ has a root. Fix $i$ and $j_1,\!\cdots\!,j_m$. $m$-inconsistency means that for any choice of $t_1,\!\cdots\!,t_m$ and $c$, one of $P_{ij_l}(t_l,c)$ is not 0. Instead of fixing $x$ and pondering at $T$, let's fix $t_1$ to $t_m$ and name $f_l(x)=P_{ij_l}(t_l,x)$. $m$-inconsistency is equivalent to saying that for any choice of $t_l$, the family $(f_l)_{1\leqslant l\leqslant m}$ of polynomials can't have a common root.
  
  Since $k$ is not AS-closed, we can find a separable polynomial $d$ with no root in $k$. Write $d(z)=r_nz^n+\cdots+r_1z+r_0$, and fix a lift $\delta(z)=\rho_nz^n+\cdots+\rho_1z+\rho_0$ to $K$. $\delta$ also has no root in $K$. Let $D(z_1,z_2)=r_nz_1^n+r_{n-1}z_1^{n-1}z_2+\cdots+r_1z_1z_2^{n-1}+r_0z_2^n$ be the homogenized version of $d$ and similarly $\Delta(z_1,z_2)$ be the homogenized version of $\delta$.
  
  Now $D(z_1,z_2)=0$ iff $z_1=0=z_2$ by the choice of $d$, and same goes for $\Delta$. Let $f,g$ be two polynomials. Then $f,g$ have a common root iff $D(f(x),g(x))$ has a root. Thus we have $m$-inconsistency in $k$ iff the family $(f_l)_{1\leqslant l\leqslant m}$ has no common root in $k$ iff $D(f_1(x),D(f_2(x),\cdots))$ has no root in $k$ iff, by $p$-henselianity, $\Delta(f_1(x),\Delta(f_2(x,\cdots))$ has no root in $K$ iff the family $(f_l)_{1\leqslant l\leqslant m}$ has no common root in $K$, the latter exactly giving $m$-inconsistency of the pattern in $K$.
\end{proof}

Thus, given an NTP2 henselian field $(K,v)$, if we take a coarsening of $v$ with residue characteristic $p$, we know its residue field has finitely many AS-extensions, without having to ponder at external definability or anything.

\subsection{Semitameness}

Recently, Franz-Viktor Kuhlmann proved in \cite{FVK-ASfin} that valued fields of characteristic $p$ with finitely many Artin-Schreier extensions are \emph{semitame}, which is a notion he studied in detail in a joint paper with Anna Rzepka. In particular, contrary to the NIP case, where AS-closure implies defectlessness, NTP2 fields could have defect, only, no \emph{dependent} defect:

\begin{dfn}
 Let $(L,w)/(K,v)$ be a purely defect Galois extension of degree $p$. Let $\sigma\in\Gal(L/K)\setminus\set{\id}$. Consider the set $\Sigma=\set{w(\tfrac{\sigma(x)-x}{x})}{x\in L^\times}$. If there is a convex subgroup $\Delta\subset\Gamma$ such that $\Sigma=\set{\gamma\in\Gamma}{\gamma>\Delta}$, we call $(L,w)/(K,v)$ an independent defect extension. Otherwise, we call it a dependent defect extension.
\end{dfn}

\begin{dfn}\label{tamedef}
 A non-trivially valued field $(K,v)$ of residue characteristic $p$ is called semitame if $\Gamma$ is $p$-divisible, $k$ is perfect, and $(K,v)$ is defectless. Valued fields of residue characteristic 0 are always called semitame. Here we will furthermore let trivially valued fields, of any characteristic, be called semitame.
\end{dfn}

Note that tame implies semitame; in fact, a valued field is tame iff it is semitame, henselian and defectless. 

Semitameness is a first-order property, though this might not be clear if defined as we did; equivalent definitions can be found in \cite{FVK-ASfin}, as well as a proof of the following result:

\begin{thm}\label{ASfin->smtm}
 Let $(K,v)$ be a valued field of equicharacteristic $p$. If $K$ is AS-finite, then $(K,v)$ is semitame.
\end{thm}

We will also need the following lemma:

\begin{lem}[{\cite[Prop.~1.4]{gdr}}]\label{tamecompo}
 A composition of two semitame henselian valuation, each of residue characteristic $p$, is semitame.
\end{lem}

Note that the statement by Kuhlmann and Rzepka that we reference is formulated for ``generalized deeply ramified'' fields (gdr) without restricting to residue characteristic $p$, and is then claimed to also hold in the semitame context; as stated, it is slightly wrong, as one needs to avoid some stupid counterexample: if $(K,v)$ is of equicharacteristic 0 with a non-divisible value group, say, $\mathds Z$, and $(k_v,w)$ is mixed-characteristic tame; then $(K,w\circ v)$ is not tame, nor semitame, because its value group is not $p$-divisible. Thus, Kuhlmann and Rzepka's proof appears to have a hidden assumption, namely, residue characteristic $p$, that we made explicit here. 

In fact, the definition of gdr fields is precisely made in order to be well behaved under composition, as well as to include finitely ramified fields which aren't tame but are still very well behaved. We will not define this notion here, instead, we refer to the aforementionned paper \cite{gdr}.

We prove a quick but very useful NTP2 version of \Cref{resperf}:

\begin{lem}\label{3.4}
 Let $K$ be NTP2, let $v$ be $p$-henselian of residue characteristic $p$, and suppose $k_v$ is imperfect; then $v$ is the coarsest valuation with residue characteristic $p$. In particular, there is at most one imperfect residue of characteristic $p$.
\end{lem}

\begin{proof}
  Suppose $w$ is a non-trivial proper coarsening of $v$ with residue characteristic $p$. Then $(k_w,\overline{v})$ is a non-trivial equicharacteristic $p$ valued field with imperfect residue. By \Cref{ASfin->smtm}, since semitame fields have residue perfect, $k_w$ is not semitame and thus has infinitely many AS-extensions. But, by AS-lifting, that means $K$ has TP2. Thus $v$ can't have any proper coarsening of residue characteristic $p$.
\end{proof}

We combine all this with the standard decomposition around $p$, written in terms of places $K\xrightarrow{v_0}k_0\xrightarrow{\overline {v_p}}k_p\xrightarrow{\overline v}k_v$ as in \Cref{standec}, and obtain:

\begin{prop}\label{AJishNTP2->}
 Let $K$ be NTP2 and $v$ be $p$-henselian, where $p=\ch(k)$. Then $(K,v)$ is either
 \begin{enumerate}
  \item of equicharacteristic $p$ and semitame, or
  \item\label{this} of mixed characteristic with $(k_0,\overline v)$ semitame, or
  \item of mixed characteristic with $v_p$ finitely ramified and $(k_p,\overline v)$ semitame.
 \end{enumerate}
 In particular, $(K,v)$ is gdr.
\end{prop}

\begin{proof}
 Most cases follow directly from \Cref{ASfin->smtm} and Artin-Schreier lifting as for the \NIPn case, we only give details for case \ref{this}.
 
 Let $(K,v)$ be of mixed characteristic such that $v_p$ is infinitely ramified, that is, $\Delta_0/\Delta_p$ is dense. This is an elementary statement, that is, going to $(K^*,v^*)\geqcur(K,v)$ sufficiently saturated and doing the standard decomposition in this new structure, $\Delta^*_0/\Delta^*_p$ remains dense; see \cite[Lem.~2.6]{AJ-NIP}. Furthermore, $(k^*_0,\overline {v^*_p})$ is defectless and has value group $\mathds R$. These facts come directly from saturation, see \cite{AK-tame}.
 
 By Artin-Schreier lifting, $k_p$ is AS-finite, and thus $(k_p,\overline v)$ is semitame. Finally, an argument similar to the aforementioned proof allows us to obtain perfection of $k_p$: going to yet another sufficiently saturated elementary extension $(L,u)$ of $(k_0,\overline{v_p})$ -- in a language of valued fields --, we know that the value group has a proper convex subgroup below $u(p)$; thus there is a non-trivial coarsening of $u$ with residue characteristic $p$, and by \Cref{3.4} $k_u$ is perfect. This is a first-order statement, so $k_p$ is also perfect.
 
 So, $(k_0,\overline{v_p})$ is defectless, has divisible value group, and perfect residue, thus it is semitame; and $(k_p,\overline v)$ is semitame. By \Cref{tamecompo}, $(k_0,\overline v)$ is semitame, as wanted.
\end{proof}

\begin{cor}
 Let $(K,v)$ be $p$-henselian, of mixed characteristic, and infinitely ramified. If $K$ is NTP2, then $(K,v)$ is roughly $p$-divisible, of perfect residue, and has no dependent defect extension.
\end{cor}

 \bibliographystyle{plain}
\bibliography{../bigbib/mtvf.bib,../bigbib/algbooks.bib,../bigbib/mt.bib,../bigbib/vfbooks.bib,../bigbib/mtbooks.bib,../bigbib/misc.bib,../bigbib/alg.bib}
\end{document}